\documentclass[12pt]{amsart}
\usepackage{amssymb}
\usepackage{amsmath}
\theoremstyle{plain}
 \newtheorem{thm}{Theorem}[section]
 
 \newtheorem{lem}[thm]{Lemma}

 \theoremstyle{definition}
 \newtheorem{defn}[thm]{Definition}
 \theoremstyle{remark}
 \newtheorem{rem}[thm]{Remark}
 \numberwithin{equation}{section}
\begin{document}

\title[Yamabe flow and the Myers-Type theorem]{Yamabe flow and the Myers-Type theorem on complete manifolds}

\author{Li Ma, Liang Cheng}

\address{Department of mathematics\\
Henan Normal University \\
Xinxiang 453007 \\
China} \email{nuslma@gmail.com} \dedicatory{}

\address{School of Mathematics and Statistics\\
 Huazhong Normal University\\
Wuhan, 430079, P.R. CHINA}

\email{math.chengliang@gmail.com} \dedicatory{}

\date{}

\begin{abstract}

In this paper,
 we prove the following Myers-type theorem: if $(M^n,g)$, $n\geq 3$, is
an n-dimensional complete locally conformally flat Riemannian
manifold with bounded Ricci curvature satisfying the Ricci
pinching condition $Rc\geq \epsilon Rg>0$, where $\epsilon>0$ is
an uniform constant, then $M^n$ must be compact.


\end{abstract}

\thanks{$^*$ The research is partially supported by the National Natural Science
Foundation of China 10631020 and SRFDP 20090002110019 }
 \maketitle

\section{Introduction}
The Yamabe flow has been proposed by R.Hamilton \cite{H89} in the
early 1980s as a tool for constructing metrics of constant scalar
curvature in a given conformal class of Riemannian metrics on a
manifold of dimension no less than three. There are some
interesting results of the Yamabe flow on closed manifolds. For
the case when the initial metric is locally conformally flat and
has positive Ricci curvature, B.Chow \cite{chow} has proved that
the normalized Yamabe flow converges to a metric of constant
curvature on closed manifolds. Then R.Ye \cite{Y94} has improved
B.Chow's result, by assuming only that the initial metric is
locally conformally flat, that the normalized Yamabe flow
converges to a metric of constant scalar curvature on closed
manifolds in this case. Recently, S.Brendle \cite{BS2005} has
proved that the normalized Yamabe flow converges to a metric of
constant scalar curvature on closed manifolds if the initial
metric is locally conformally flat or $3\leq n\leq 5$, where $n$
is the dimension of the manifold. For other recent works of Yamabe
flow on closed manifolds, one may see \cite{BS2007} and \cite{SS}.

Precisely, in this paper, we consider the Yamabe flow $(M^n,g(t)),
t\in (0,T)$, satisfying
\begin{equation}\label{yamabe}
\frac{\partial g}{\partial t}=-Rg
\end{equation}
on a complete locally conformally flat Riemannian manifold
$(M^n,g(0))$, $n\geq 3$, where $g(t)$ evolves in the conformal
class of the given metric $g(0)$, i.e.
$g(t)=u(t)^{\frac{4}{n-2}}g(0)$ for $u(t)$ being a positive smooth
function on $M^n$, and $R$ is the scalar curvature of $g(t)$.

By using the Ricci flow and the Yamabe flow as tools, there are
some interesting Myers-type results, which state that if a
complete manifold with bounded curvature satisfies a pinching
condition, then this manifold must be compact. B.L.Chen and
X.P.Zhu \cite{CZ2000}, by using Ricci flow, have firstly proved
that if $M^n$ is a complete n-dimensional Riemannian manifold with
positive and bounded scalar curvature and satisfies the pinching
condition
$$
|W|^2+|V|^2\leq \delta_n(1-\epsilon)|U|^2,
$$
where $W$, $V$, $U$ are the Weyl part, scalar curvature part and
traceless Ricci part of the Riemannian curvature tensor
respectively and $\delta_4=\frac{1}{5}$, $\delta_5=\frac{1}{10}$,
$\delta_n=\frac{2}{(n-2)(n+1)}, n\geq 6$, $\epsilon>0$ is a
constant. When $n=3$, they have also proved if $M^3$ is
3-dimensional complete Riemannian manifold with bounded
and nonnegative sectional curvature satisfying the Ricci pinching
condition
$$
Rc\geq \epsilon Rg>0,
$$
where $\epsilon>0$ is a constant, then $M^n$ must be compact. Then
L.Ni and B.Q.Wu \cite{NW} have improved B.L.Chen and X.P.Zhu's
result for $n\geq 4$ and have proved that if $M^n$ is an
n-dimensional complete Riemannian manifold with bounded curvature
operator satisfying the pinching condition
$$
Rm\geq \delta R_I>0,
$$
where $\delta>0$ is a constant and $R_I$ is the scalar part of the
curvature operator $Rm$, then $M^n$ must be compact. In the recent
interesting work \cite{BS}, S.Brendle and R. Schoen have proved if
$M^n$, $n\geq 4$, is an n-dimensional complete Riemannian manifold
with bounded curvature operator satisfying the pinching condition
\begin{eqnarray*}
&&Rm(e_1,e_3,e_1,e_3)+\lambda^2Rm(e_1,e_4,e_1,e_4)+
\mu^2Rm(e_2,e_3,e_2,e_3)\\
&&+\lambda^2\mu^2Rm(e_2,e_4,e_2,e_4)-2\lambda\mu
Rm(e_1,e_2,e_3,e_4)\geq \delta R>0
\end{eqnarray*}
for all orthonormal four frames $\{e_1,e_2,e_3,e_4\} $ and all
$\lambda,\mu\in [-1,1]$, where $\delta>0$ is a constant and $R$ is
the scalar curvature, then $M^n$ must be compact.
In \cite{GH}, H.L.Gu has proved
that, by using the Yamabe flow, if $(M^n,g)$, $n\geq 3$, is an
n-dimensional complete locally conformally flat Riemannian
manifold with bounded Ricci curvature and nonnegative sectional
curvature satisfying the Ricci pinching condition
$$
Rc\geq \epsilon
Rg>0,
$$
 where $\epsilon>0$ is an uniform constant, then $M^n$ must
be compact.

We remark that all the results above need the condition that $M^n$
has nonnegative sectional curvature (or positive curvature
operator). This assumption gives a nature injectivity bound needed
in the convergence theorem, which also needed to apply dimension
reduction techniques (see \cite{RH2}). Without this assumption, we
meet the difficult conjecture of R.Hamilton,
 which states that if $M^3$ is a 3-dimensional complete
Riemannian manifold satisfying the Ricci pinching condition
$$Rc\geq \epsilon Rg>0,$$
where $\epsilon>0$ is an uniform constant,
then $M^3$ must be compact (see \cite{BC}).

In this paper, we prove the following result which partially
confirms the conjecture of R.Hamilton on 3-dimensional locally
conformally flat manifolds.

\begin{thm}\label{main}
If $(M^n,g)$, $n\geq 3$, is an n-dimensional complete locally
conformally flat Riemannian manifold with bounded Ricci curvature
satisfying the pinching condition
$$Rc\geq \epsilon Rg>0,$$
where $\epsilon>0$ is an uniform constant, then $M^n$ must be
compact.
\end{thm}

\begin{rem}
Note that by the strong maximum principle, Theorem \ref{main} is
equivalent to say
 that if $(M^n,g)$, $n\geq 3$, is an n-dimensional
complete noncompact locally conformally flat Riemannian manifold
with bounded Ricci curvature satisfying the Ricci pinching condition
$Rc\geq \epsilon Rg$ and $R\geq 0$, where $\epsilon>0$ is an uniform
constant, then $M^n$ must be flat.
\end{rem}

\section{preliminaries}
  We shall recall some basic formulae and convergence results
  concerning with the Yamabe flows. Similar results for the Ricci
  flow are well-known.

We first recall some formulae from the fundamental paper
\cite{chow} (see Lemmas 2.2 and 2.4).

\begin{lem} If $(M^n,g(t))$, $n\geq 3$, is the solution to the Yamabe flow (\ref{yamabe}) on
an n-dimensional complete locally conformally flat Riemannian
manifold, then
\begin{equation}\label{scalar}
R_t=(n-1)\Delta R+R^2 \end{equation}
 and
\begin{equation}\label{ric}
\partial_t R_{ij}=(n-1)\Delta R_{ij}+\frac{1}{n-2}B_{ij},
\end{equation}
where
$$
B_{ij}=(n-1)|Ric|^2g_{ij}+nRR_{ij}-n(n-1)R_{ij}^2-R^2g_{ij}.
$$
\end{lem}

As pointed out in \cite{DN}, we can rewrite the equation
(\ref{ric}) for $Rc$ as
$$
\partial_t Rc = (n-1)\Delta Rc + Rc\ast Rc,
$$
where $Rc\ast Rc$ stands for any linear combination of tensors
formed by contraction on $R_{ij}\cdot R_{kl}$. Notice that the
evolution for $Rc$ along the Yamabe flow has the same form as the
evolution for $Rm$ along the Ricci flow. Techniques similar to
Shi's in \cite{S97} can be applied to the Yamabe flow as well, and
we can show that all the covariant derivatives of $Rc$ are locally
uniformly bounded on $(0,T)$ if $|Rc|$ is bounded on $[0,T)$.
Hence, all the covariant derivatives of the Riemannian curvature
$Rm$ are uniformly bounded on $(0,T)$ if $|Rc|$ is bounded on
$[0,T)$ in the locally conformally flat Riemannian manifolds under
the Yamabe flow. Similarly, techniques similar to Shi's in
\cite{S89} can also be applied to the Yamabe flow, and we have
that there exists a solution to the Yamabe flow on noncompact
locally conformally flat Riemannian manifolds with bounded Ricci
curvature in some time interval $[0,T)$ (see \cite{AM}).

It is well known that the singularity analysis plays the important
role in the study of geometric flows. To study the singularities
of geometric flows, one often dilates about a singularity based on
the blow up rate of the curvature. Note that the curvature bound
is immediately satisfied for the blow up about a singularity, but
the injectivity radius bound is not. Especially, it seems hard to
get the injectivity radius bound without extra conditions along
the geometric flows in noncompact manifolds, since the injectivity
radius may not has uniformly lower bound at the initial time. In
order to handle this problem, we shall use the method firstly
proposed by K.Fukaya \cite{F} in metric geometry and later by
D.Glickenstein \cite{DG} to the Ricci flow. The latter work gives
a kind of precompactness theorem (\cite{DG}, Theorem 3) of the
Ricci flow without injectivity radius estimates. We note that
Fukaya-Glickenstein's theorem also holds for the Yamabe flow. The
reasons are below.

First we need an elementary fact in Riemannian geometry and one
can find a proof in \cite{MC}: Let $(M^n, g)$ be a complete
Riemannian manifold with bounded sectional curvature $|sec| \leq
1$. Given a point $p \in M^n$. Denote the exponential map by
$exp_p:T_p M^n \to M^n$ with $B(o,\pi) \subset T_p M^n$ equipped
with metric $exp_p^*g$. Then the injectivity radius at $o$ in
$B(o,\pi)$ has its lower bound that $\text{inj}(o)>\frac{\pi}{2}$.

So we have the following precompactness theorem for the Yamabe
flow.
\begin{lem}\label{precompactness}
 Let $\{(M^n_i,g_i(t),x_i)\}_{i=1}^{\infty}$, $t\in [0,T]$, be a sequence of
 the Yamabe flows on the complete locally conformally flat Riemannian manifolds
such that
$$\sup\limits_{M^n_i\times [0,T]}|Rm(g_i(t))|_{g_i(t)}\leq 1.$$ Let
$\phi_i=exp_{x_i,g_i(0)}$ be the exponential map with respect to
metric $g_i(0)$ and $B(o_i,\frac{\pi}{2})\subset T_{x_i}M_i$
equipped with metric $\widetilde{g}_i(t)\triangleq \phi_i^*g(t)$.
Then $(B(o_i,\frac{\pi}{2}),\widetilde{g}_i(t),o_i)$ subconverges to
a Yamabe flow $(B(o,\frac{\pi}{2}),\widetilde{g}(t),o)$ in
$C^{\infty}$ sense, where $B(o,\frac{\pi}{2})\subset \mathbb{R}^n$
equipped with metric $\widetilde{g}(t)$.
\end{lem}
\begin{proof}
Since $\widetilde{g}_i(t)\triangleq \phi_i^*g(t)$, we have
$$\sup\limits_{B(o_i,\frac{\pi}{2})\times
[0,T]}|Rm(\widetilde{g}_i(t))|_{\widetilde{g}_i(t)}\leq 1.$$  As
we mentioned before, Shi's local derivative estimates of curvature
operator also hold for the Yamabe flow on complete locally
conformally flat manifolds. Note that we also have
$inj(o_i,\widetilde{g}_i(0))>\frac{\pi}{2}$. Then the result
follows from the proof of Hamilton's precompactness theorem for
the Ricci flow (see \cite{RH}).
\end{proof}

Next we recall the following two definitions in \cite{DG}.

\begin{defn}\cite{DG}\label{defn1}
A sequence of pointed n-dimensional Riemannian manifolds $\{(M^n_i,
 g_i,x_i )\}_{i=1}^{\infty}$ locally converges to a pointed metric space $(X,d,x)$ in
the sense of $C^{\infty}$-local submersions at $x$ if there is a
Riemannian metric $\overline{g}$ on an open neighborhood $V\subset
\mathbb{R}^n$ of $o$, a pseudogroup $\Gamma$ such that the quotient
is well defined, an open set $U\subset X$, and maps $\varphi_i: (V,
o) \to (M_i,x_i )$ such that

(1) $\{(M^n_i,g_i,x_i )\}_{i=1}^{\infty}$ converges to $(X,d, x)$ in
the pointed Gromov-Hausdorff distance,

(2) the identity component of $\Gamma$ is a Lie group germ,

(3) $(V/\Gamma,\bar{d}_{\overline{g}})$ is isometric to $(U,d)$,
where $\bar{d}_{\overline{g}}$ is the induced distance in the
quotient,

(4) $(\varphi_i)_{*}$ is nonsingular on $V$ for all $i\in N$, and

(5) $\overline{g}$ is the $C^{\infty}$ limit of $\varphi^{*}_ig_i$
(uniform convergence on compact sets together with all derivatives).
\end{defn}

\begin{defn}\cite{DG}
A sequence of pointed n-dimensional Riemannian manifolds $\{(M^n_i,
 g_i,x_i )\}_{i=1}^{\infty}$ converges to a pointed metric
space $(X, d, x)$ in the sense of $C^{\infty}$-local submersions if
for every $y\in X$ there exist $y_i  \in M_i$ such that
$\{(M^n_i,g_i,y_i )^{\infty}_{i=1}\}$ locally converges to $(X,d,y)$
in the sense of $C^{\infty}$-local submersions at $y$.
\end{defn}

Note that there exists subsequence $\{( M_{i_k}, g_{i_k}(t),x_{i_k})
\}_{k=1}^{\infty}$
 converges to $(X,d (t) , x)$ for each $t\in [0, T]$ in Gromov-Hausdorff distance
 by $|Rm(g_i(t))|_{g_i(t)}\leq 1$ and Theorem 19 in \cite{DG}. In
 fact
$\phi_i=exp_{x_i,g_i(0)}$ in Lemma \ref{precompactness} defines a
'locally' covering map between $B(o_i,\frac{\pi}{2})\subset
T_{x_i}M_i$ and $B(x_i,\frac{\pi}{2})\subset M_i$. This defines
pseudogroups $\Gamma_i$ acts isometrically on on
$B(o,\frac{1}{4})$(see \cite{F}, P.9 or \cite{DG}, \S 5).
Furthermore, $\Gamma_i$ converge to a limit pseudogroup $\Gamma$
(see see \cite{F}, P.9) such that $(B (o_i, \frac{1}{4}) ,
\Gamma_i)$, where $B(o, \frac{1}{4})\subset B(o_i,\frac{\pi}{2})$,
converges to $(B (o, \frac{1}{4}) ,\Gamma)$ in the equivariant
Gromov-Hausdorff distance(see \cite{F}, Definition 1.9), and hence
$B (o_i, \frac{1}{4}) /\Gamma_i$ converges to $B (o,
\frac{1}{4})/\Gamma$ in the Gromov-Hausdorff distance (see \cite{F},
Lemma 1.11). Since $B (o_i, \frac{1}{4})/\Gamma_i$ is isometric to a
neighborhood of $x_i $, $B (o, \frac{1}{4}) /\Gamma$ is isometric to
a neighborhood of $x$. Note that $B (o_i, \frac{1}{4})$ converges to
$B (o, \frac{1}{4})$ in $C^{\infty}$ sense by Lemma
\ref{precompactness}. So we have proved that $ (M_{i},d_{g_{i}(t)},
x_{i})$ converges to $(X,d (t) , x)$ in the sense of
$C^{\infty}$-local submersions at x. If we identify
$B(o_i,\frac{1}{4})$ with $B(o,\frac{1}{4})$ by map $id$, then
$\varphi_i=\phi_i\circ id$, where $\varphi_i$ is defined in
Definition \ref{defn1}. Note that $\Gamma$ is a Lie group germ by
\cite{F}, \S3. Just notice that if $(M_{i},d_{g_{i}(t)}, x_{i})$
converges to $(X,d (t) , x)$ in the pointed Gromov- Hausdorff
distance, then for every $y \in X$ there exist $y_i \in X_i$ such
that $(M_{i},d_{g_{i}(t)}, y_{i})$ converges to $(X,d (t) , x)$ in
the pointed Gromov-Hausdorff distance(see \cite{DN}, Proposition
12). Then $ (M_{i},d_{g_{i}(t)}, x_{i})$ converges to $(X,d (t) ,
x)$ in the sense of $C^{\infty}$-local submersions.

\begin{thm}\label{Yamabe_precompactness}
 Let $\{(M^n_i,
 g_i(t),x_i )\}_{i=1}^{\infty}$,
where $t\in [0,T]$, be a sequence of pointed solutions to the
Yamabe flows on the locally conformally flat Riemannian manifolds
such that
$$\sup\limits_{M^n_i\times [0,T]}|Rm(g_i(t))|_{g_i(t)}\leq 1,$$ and
for all $i\in N$ and $t\in [0, T]$.

Then there is a subsequence which still denote by $\{( M_{i},
g_{i}(t),x_{i} )\}_{i=1}^{\infty}$ and a one parameter family of
complete pointed metric spaces $(X,d(t),x)$ such that for each $t\in
[0, T]$, $ (M_{i},d_{g_{i}(t)}, x_{i})$ converges to $(X,d (t) , x)$
in the sense of $C^{\infty}$-local submersions and the metric
$\overline{g}(t)$ in definition \ref{defn1} is solution to the
Yamabe flow.
\end{thm}
\begin{rem}
In fact, Fukaya-Glickenstein's theorem holds for any sequence of
geometric flows $\{(M^n_i, g_i(t),x_i )\}_{i=1}^{\infty}$,
$\frac{\partial g}{\partial t}=h(g)$ satisfying $h(g_i(t))<C$,
$|Rm(g_i(t))|_{g_i(t)}\leq C$ on $[0,T)$ and Shi's local
derivative estimates of curvature operators hold.
\end{rem}

Note that the limit space $(X,d(t))$ is an Alexandrov space, since
the sectional curvature of $M_{i}$ has a uniformly lower bound.
Finally, we need the following theorem which can be found in
\cite{BGP92}, Theorem 3.6.

\begin{thm}\cite{BGP92}\label{Alexandrov}
Let M be a complete Alexandrov space with curvature $>K$, $K > 0$.
Then $diam (M) \leq \frac{\pi}{\sqrt{K}}$.
\end{thm}

\section{Singularity model of the Yamabe flow}
 Recall that R.Hamilton\cite{RH2} has proposed the singularity models
 which classify all the maximal solutions to Ricci flow into three types. We note that the same
 classification can be applied to the Yamabe flow; every maximal
 solution to the Yamabe flow on locally conformally flat
  manifolds with nonnegative Ricci curvature is of only one of the following three types:
\begin{defn}
Suppose that $(M^n,g(t))$ is a solution to the Yamabe flow on a
locally conformally flat manifold with nonnegative Ricci curvature.
If $T<\infty$, we say that the solution forms a
\begin{enumerate}
\item Type I singularity if $\sup\limits_{M\times
[0,T)}(T-t)R<\infty$,

\item Type IIa singularity if $\sup\limits_{M\times
[0,T)}(T-t)R=\infty$.

\end{enumerate}
Similarly, if $T=\infty$, we say that the solution forms a
\begin{enumerate}
\item Type IIb singularity if $\sup\limits_{M\times
[0,\infty)}tR=\infty$,

\item Type III singularity if $\sup\limits_{M\times
[0,\infty)}tR<\infty$.

\end{enumerate}
\end{defn}

For any maximum solution to the Yamabe flow on a locally conformally
flat manifold with nonnegative Ricci curvature, we have that if the
infimum of injectivity radius $\rho(t)$ at all points satisfies
$\rho(t)\geq \frac{c}{\sqrt{M(t)}}$, where $c>0$ is a uniform
constant and $M(t)$ denotes the supremum of the curvature at time
$t$, then there exists a sequence of dilations of the solution which
converges in the limit to one of the following singularity model of
the corresponding type (see \cite{RH2}) in the sense of Definition
\ref{limit_solutions} below.

\begin{defn}\label{limit_solutions}
Suppose that $(M^n,g(t))$ is a limit solution to the Yamabe flow on
a locally conformally flat manifold with nonnegative Ricci
curvature. We say that the limit solution is
\begin{enumerate}

\item  Type I limit solution if it exists for
 $-\infty<t<\Omega$ for some constant $\Omega$ with $0<\Omega<+\infty$ and $R\leq
\frac{\Omega}{\Omega-t}$ everywhere with equality holds somewhere at
$t=0$,

\item Type II limit solution if exists for
 $-\infty<t<+\infty$ and $R\leq
 1$ everywhere with equality holds somewhere at $t=0$.

\item Type III limit solution if it exists for
 $-A<t<+\infty$ for some constant $A$ with $0<A<+\infty$ and $R\leq
 \frac{A}{A+t}$ with equality holds somewhere at $t=0$.
\end{enumerate}
 \end{defn}

As we mentioned before, the injectivity radius lower bound may not
be available in the sequence of dilations of the maximal solution
to the Yamabe flow on complete and noncompact manifolds. Hence the
singularity model in Definition \ref{limit_solutions} may not
suitable for our original Yamabe flow. However, we shall show how
to use Theorem \ref{Yamabe_precompactness} to avoid the assumption
of uniform injectivity radius bound in the next section.

 In order to
prove Theorem \ref{main}, we need a local version of a result
proved by H.L.Gu \cite{GH}, which is based on the B.Chow's Harnack
inequality \cite{chow}.

\begin{thm} \label{expanding_soliton}
Let $D\subset M^n$ be a simply connected open domain of a complete
n-dimensional locally conformally flat Riemannian manifold such
that B.Chow's Harnark inequality and the strong maximum principle
for the Harnark quantity $Z$  of the Yamabe flow (see
(\ref{Benchow's harnack})) hold true on $D$. Then any Type III
limit solution with positive Ricci curvature to the Yamabe flow on
$D\subset M^n$ is necessarily a homothetically expanding gradient
soliton.
\end{thm}

\begin{proof} We follow the argument in \cite{GH} and assume that $D=M$ without loss of generality.
We may assume that, after a shift of the time variable, the Type
III limit solution of the Yamabe flow on locally conformally flat
manifolds is defined for $0<t<+\infty$, where $tR$ achieves its
maximum in space-time. Recall that B.Chow \cite{chow} has proved
the following Harnack inequality
\begin{equation}\label{Benchow's harnack}
Z=\frac{\partial R}{\partial t}+<\nabla R,
X>+\frac{1}{2(n-1)}R_{ij}X^iX^j+\frac{R}{t}\geq 0,
\end{equation}
for the Yamabe flow on the closed locally conformally flat manifolds
with positive Ricci curvature. We remark that by the same proof and
by using the maximum principle, this Harnack inequality clearly
holds for the Yamabe flow on the complete locally conformally flat
manifolds with nonnegative and bounded Ricci curvature.

Since $tR$ achieves its maximum at some $(x_0,t_0)$, (\ref{Benchow's
harnack}) vanishes in the direction $X=0$ at $(x_0,t_0)$. By the
strong maximum principle (see \cite{DN}, Lemma 3.2), we know that at
any $t<t_0$ and any point $x\in M^n$, there is a vector $X\in T_x
M^n$ such that $Z=0$. Take the first variation of $Z$ in $X$, we get
\begin{equation}\label{harnack1}
\nabla_i R+\frac{1}{n-1}R_{ij}X^j= 0.
\end{equation}
We remark that for $(R_{ij})>0$, the equation above uniquely
determines a vector field $X$.

Substituting (\ref{harnack1}) into $Z=0$, we have
\begin{equation}\label{harnack2}
\frac{\partial R}{\partial t}+\frac{R}{t}+\frac{1}{2}\nabla_i R
\cdot X^i=0.
\end{equation}
We now denote $\partial_t-(n-1)\Delta$ by $\Box$. Applying
$\frac{1}{2}X^i\Box$ to (\ref{harnack1}), $\Box$ to (\ref{harnack2})
and then take the sum, we have
\begin{align}\label{harnack3}
&X^i \Box (\nabla_i R)+\frac{1}{2(n-1)}X^iX^j\Box R_{ij}-\nabla_k
R_{ij}(\nabla_k X^j)X^i\nonumber\\
&-(n-1)\nabla_k\nabla_iR\cdot \nabla_k X^i+\Box(\frac{\partial
R}{\partial t}+\frac{R}{t})=0.
\end{align}
We also have
\begin{align}\label{harnack4}
\Box (\nabla_i R)&=\nabla_i (\Box R)-(n-1)R_{il}\nabla_l
R\nonumber\\
&=\nabla_i (R^2)-(n-1)R_{il}\nabla_l R.
\end{align}
By Lemma 3.8 in \cite{chow}, we get
\begin{align}\label{harnack5}
\Box(\frac{\partial R}{\partial t}+\frac{R}{t})=&3(n-1)R\Delta R+\frac{1}{2}(n-1)(2-n)|\nabla R|^2\nonumber\\
&+2R^3+\frac{R^2}{t}-\frac{R}{t^2}.
\end{align}
Substituting (\ref{ric}), (\ref{harnack4}) and (\ref{harnack5}) into
(\ref{harnack3}), we get
\begin{align}\label{harnack6}
&X^i (\nabla_i (R^2)-(n-1)R_{il}\nabla_l
R)+\frac{1}{2(n-1)(n-2)}X^iX^j B_{ij}\nonumber\\
&-\nabla_k R_{ij}(\nabla_k X^j)X^i -(n-1)\nabla_k\nabla_iR\cdot
\nabla_k X^i\nonumber\\
&+3(n-1)R\Delta R+\frac{1}{2}(n-1)(2-n)|\nabla R|^2\nonumber\\
&+2R^3+\frac{R^2}{t}-\frac{R}{t^2}=0.
\end{align}
It follows from (\ref{harnack1}) that
\begin{align}\label{harnack7}
\nabla_k\nabla_i R+\frac{1}{n-1}(\nabla_k
R_{ij})X^i=-\frac{1}{n-1}R_{ij}\nabla_k X^i,
\end{align}
and
\begin{align}\label{harnack8}
X^iR_{il}\nabla_l R+\frac{1}{n-1}R_{il}R_{jl}X^iX^j=0.
\end{align}
We also have
\begin{align}\label{harnack9}
Z=(n-1)\Delta R+<\nabla R,
X>+\frac{1}{2(n-1)}R_{ij}X^iX^j+R^2+\frac{R}{t}=0.
\end{align}
Substituting (\ref{harnack7}), (\ref{harnack8}) and (\ref{harnack9})
into (\ref{harnack6}), we get
\begin{align}\label{harnack9}
&-R(R+\frac{1}{t})^2+\frac{1}{2(n-1)(n-2)}B_{ij}X^iX^j-\frac{1}{2(n-1)}RR_{ij}X^iX^j\nonumber\\
&+\frac{n}{2(n-1)}R_{il}R_{jl}+R_{ij}\nabla_k X^i \nabla_k X^j=0.
\end{align}
By (\ref{harnack1}), we have
$$
\nabla_k\nabla_i R=-\frac{1}{n-1}(X^j\nabla_k R_{ij}+R_{ij}\nabla_k
X^j),
$$
and then by taking the trace and using the evolution equation of
scalar curvature, we get
\begin{align}\label{harnack10}
R_{ij}((R+\frac{1}{t})g_{ij}-\nabla_iX^j)=0.
\end{align}
By (\ref{harnack9}) and (\ref{harnack10}), we conclude that
\begin{align}\label{harnack11}
R_{ij}(\nabla_k X^i-(R+\frac{1}{t})g_{ik})(\nabla_k
X^j-(R+\frac{1}{t})g_{jk})+A_{ij}X^i X^j=0,
\end{align}
where
$A_{ij}=\frac{1}{2(n-1)(n-2)}B_{ij}+\frac{1}{2(n-2)}(nR_{il}R_{jl}-RR_{ij})$.

Then in local coordinates where $g_{ij}=\delta_{ij}$ and the Ricci
tensor $(R_{ij})$ is diagonal, we have
$$
\sum\limits_{i}\lambda_i(\nabla_k
X^i-(R+\frac{1}{t})g_{ik})^2+A_{ij}X^iX^j=0.
$$
By \cite{chow}(3.13), we have
$\nu_i=\frac{1}{2(n-1)(n-2)}\sum\limits_{k,l\neq i,
k>l}(\lambda_k-\lambda_l)^2$, where $\nu_i$ is the eigenvalue of
$A_{ij}$. Since $\lambda_i> 0$, the theorem holds immediately.
\end{proof}

Finally, we need the following
\begin{thm}\cite{GH}\label{pinching_expanding}
There exists no noncompact locally conformally flat Type III limit
solution of the Yamabe flow which satisfies the Ricci pinching
condition
$$Rc \geq \epsilon R g > 0,$$
for some constant $\epsilon >0$.
\end{thm}

\section{Pinching estimates}
In \cite{chow}, B.Chow proved the inequality $R_{ij}\geq \epsilon
Rg_{ij}>0$ is preserved under the Yamabe flow on compact locally
conformally flat manifolds. Clearly his proof also works in the
complete setting all curvature operator are uniformly bounded in
space, at each time-slice, which can apply the maximum principle
for complete manifolds. B.Chow \cite{chow} also gets the pinching
estimate that $Rc_{max}-Rc_{min}\leq CR^{1-n\epsilon}$ (so
$|Rc-\frac{1}{n}Rg|\leq CR^{1-n\epsilon}$) under Yamabe flow if
$R_{ij}\geq \epsilon R g_{ij}>0$ holds, where $C$ is a constant
only depending on $g(0)$. But this pinching estimate may not
strong enough for our purpose. In this section, we calculate the
term $\frac{|Rc|^2-\frac{1}{n}R^2}{R^{2-\delta}}$ directly and get
an improved pinching estimate.

\begin{lem}\label{pinching_equ}
If $(M^n,g(0))$, $n\geq 3$, is an n-dimensional locally conformally
flat complete Riemannian manifold and bounded Ricci curvature, then
the following equality holds for any constant $\delta$ under the
Yamabe flow (\ref{yamabe}),
\begin{eqnarray}
(\partial_t-(n-1)\Delta) f&=&\frac{2(1-\delta)(n-1)}{R}<\nabla
f,\nabla R>\nonumber\\
&&-\frac{2(n-1)}{R^{4-\delta}}|R\nabla Rc-\nabla R  Rc |^2\nonumber\\
&&-\frac{(1-\delta)\delta(n-1)}{R^{4-\delta}}(|Rc|^2-\frac{1}{n}R^2)|\nabla
R|^2\nonumber\\
&&+\frac{1}{R^{2-\delta}}(\delta
R(|Rc|^2-\frac{1}{n}R^2)-J)\nonumber,
\end{eqnarray}
where $f=\frac{|Rc|^2-\frac{1}{n}R^2}{R^{2-\delta}}$ and
$J=\frac{2}{n-2}(n(n-1)tr(Rc^3)+R^3-(2n-1)R|Rc|^2)$.
\end{lem}
\begin{proof}
By (\ref{ric}) and $|Rc|^2=g^{ik}g^{jl}R_{ij}R_{kl}$, we have
\begin{eqnarray*}
\partial_t|Rc|^2&=&2g^{ik}g^{jl}(\partial_tR_{ij})R_{kl}+2Rg^{ik}g^{jl}R_{ij}R_{kl}\\
&=&(n-1)\Delta |Rc|^2-2(n-1)|\nabla
Rc|^2+6\frac{n-1}{n-2}R|Rc|^2\\
&&-\frac{2}{n-2}R^3-\frac{2n(n-1)}{n-2}tr(Rc^3).
\end{eqnarray*}
From (\ref{scalar}), we get
\begin{eqnarray*}
\partial_t R^2=(n-1)\Delta R^2-2(n-1)|\nabla R|^2+2R^3.
\end{eqnarray*}
Hence
\begin{eqnarray*}
\partial_t(|Rc|^2-\frac{1}{n}R^2)
&=&(n-1)\Delta (|Rc|^2-\frac{1}{n}R^2)-2(n-1)(|\nabla
Rc|^2-\frac{1}{n}|\nabla R|^2)\\
&&+6\frac{n-1}{n-2}R|Rc|^2-(\frac{2}{n-2}+\frac{2}{n})R^3-\frac{2n(n-1)}{n-2}tr(Rc^3).
\end{eqnarray*}
Now we denote $\partial_t-(n-1)\Delta$ by $\Box$. So we have
\begin{eqnarray*}
\Box f &=&\frac{\Box(|Rc|^2-\frac{1}{n}R^2)}{R^{2-\delta}}
-(2-\delta)\frac{|Rc|^2-\frac{1}{n}R^2}{R^{3-\delta}}\Box R\\
&&-(2-\delta)(3-\delta)(n-1)\frac{|Rc|^2-\frac{1}{n}R^2}{R^{4-\delta}}|\nabla
R|^2\\
&&+\frac{2(2-\delta)(n-1)}{R^{3-\delta}}<\nabla
R, \nabla (|Rc|^2-\frac{1}{n}R^2)> \\
&\doteq&A+B,
\end{eqnarray*}
where
\begin{eqnarray*}
A&\doteq&-\frac{2(n-1)}{R^{2-\delta}}(|\nabla
Rc|^2-\frac{1}{n}|\nabla
R|^2)\\
&&-(2-\delta)(3-\delta)(n-1)\frac{|Rc|^2-\frac{1}{n}R^2}{R^{4-\delta}}|\nabla
R|^2\\
&&+\frac{2(2-\delta)(n-1)}{R^{3-\delta}}<\nabla R,
\nabla(|Rc|^2-\frac{1}{n}R^2)>
\end{eqnarray*}
contains the gradient terms and
\begin{eqnarray}\label{B}
B&\doteq&\frac{1}{R^{2-\delta}}(6\frac{n-1}{n-2}R|Rc|^2-(\frac{2}{n-2}+\frac{2}{n})R^3-\frac{2n(n-1)}{n-2}tr(Rc^3))\nonumber\\
&&-(2-\delta)\frac{|Rc|^2-\frac{1}{n}R^2}{R^{3-\delta}}R^2\nonumber\\
 &=&\frac{1}{R^{2-\delta}}(\delta(|Rc|^2-\frac{1}{n}R^2)R-J)
\end{eqnarray}
contains the curvature terms. We rewrite $A$ as
\begin{eqnarray*}
\frac{A}{n-1}&=&-\frac{2 }{R^{2-\delta}}(|\nabla
Rc|^2-\frac{1}{n}|\nabla R|^2)-(2-\delta)(3-\delta)
\frac{|Rc|^2-\frac{1}{n}R^2}{R^{4-\delta}}|\nabla
R|^2\\
&&+\frac{2(2-\delta)}{R^{3-\delta}} <\nabla R, \nabla(|Rc|^2-\frac{1}{n}R^2)>\\
&=&-\frac{2 }{R^{2-\delta}}(|\nabla Rc|^2-\frac{1}{n}|\nabla
R|^2)-(2-\delta)(3-\delta)
\frac{|Rc|^2-\frac{1}{n}R^2}{R^{4-\delta}}|\nabla
R|^2\\
&&+\frac{2(1-\delta)}{R^{3-\delta}} <\nabla R,
\nabla(|Rc|^2-\frac{1}{n}R^2)>\\
&&+\frac{2}{R^{3-\delta}} <\nabla R,
\nabla(|Rc|^2-\frac{1}{n}R^2)>
\end{eqnarray*}
Since
$$
\nabla(\frac{|Rc|^2-\frac{1}{n}R^2}{R^{2-\delta}})=\frac{\nabla(|Rc|^2-\frac{1}{n}R^2)}{R^{2-\delta}}
-(2-\delta)\frac{|Rc|^2-\frac{1}{n}R^2}{R^{3-\delta}}\nabla R,
$$
we get
\begin{align*}
\frac{A}{n-1} &=-\frac{2 }{R^{2-\delta}}(|\nabla
Rc|^2-\frac{1}{n}|\nabla R|^2)-(2-\delta)(1+\delta)
\frac{|Rc|^2-\frac{1}{n}R^2}{R^{4-\delta}}|\nabla
R|^2\\
&+\frac{2(1-\delta)}{R} <\nabla R,
\nabla(\frac{|Rc|^2-\frac{1}{n}R^2}{R^{2-\delta}})>\\
&+\frac{2}{R^{3-\delta}}
<\nabla R, \nabla(|Rc|^2-\frac{1}{n}R^2)>.
\end{align*}
Note that
$$
-\frac{2}{R^{2-\delta}}|\nabla
Rc|^2-2\frac{|Rc|^2}{R^{4-\delta}}|\nabla
R|^2+\frac{2}{R^{3-\delta}}<\nabla R, \nabla
|Rc|^2>=-\frac{2}{R^{4-\delta}}|R\nabla Rc-\nabla R  Rc |^2,
$$
so we have
\begin{align}\label{A}
\frac{A}{n-1} &=\frac{2(1-\delta)
}{R}<\nabla(\frac{|Rc|^2-\frac{1}{n}R^2}{R^{2-\delta}}),\nabla R>
-\frac{2 }{R^{4-\delta}}|R\nabla Rc-\nabla R  Rc |^2\nonumber\\
&-\frac{(1-\delta)\delta
}{R^{4-\delta}}(|Rc|^2-\frac{1}{n}R^2)|\nabla R|^2.
\end{align}
Combining with (\ref{A}) and (\ref{B}), we conclude Lemma
\ref{pinching_equ}.
\end{proof}

Next we need the following lemma to control the term $J$ defined in
Lemma \ref{pinching_equ}.
\begin{lem}\label{pinching_ineq}
If $(M^n,g)$, $n\geq 3$, is an n-dimensional complete locally
conformally flat Riemannian manifold and bounded Ricci curvature
satisfying $Rc\geq \epsilon Rg>0$, then we have the following
inequality holds
\begin{eqnarray*}
J\geq \frac{4}{3}n\epsilon R(|Rc|^2-\frac{1}{n}R^2),
\end{eqnarray*}
where $J$ is defined in Lemma \ref{pinching_equ}.
\end{lem}
\begin{proof}
 Let $\lambda_i$ be the
eigenvalues of $Rc$ and assume $\lambda_n\geq \lambda_{n-1}\geq
\cdots \geq \lambda_1$. Then we compute
\begin{align*}
\frac{n-2}{2}J=&n(n-1)\sum\limits_{i}\lambda_i^3+
(\sum\limits_{i}\lambda_i)((\sum\limits_{i}\lambda_i)^2-(2n-1)\sum\limits_{i}\lambda_i^2)\\
=&n(n-1)\sum\limits_{i}\lambda_i^3+
(\sum\limits_{i}\lambda_i)(-2(n-1)\sum\limits_{i}\lambda_i^2+2\sum\limits_{i<j}\lambda_i\lambda_j)\\
=&n(n-1)\sum\limits_{i}\lambda_i^3-2(n-1)(\sum\limits_{i}\lambda_i^3+\sum\limits_{i<j}\lambda_i^2\lambda_j+\sum\limits_{i<j}\lambda_i\lambda_j^2)\\
&+2(\sum\limits_{i<j}\lambda_i^2\lambda_j+\sum\limits_{i<j}\lambda_i\lambda_j^2+3\sum\limits_{i<j<k}\lambda_i\lambda_j\lambda_k)\\
=&(n-1)(n-2)\sum\limits_{i}\lambda_i^3-2(n-2)(\sum\limits_{i<j}\lambda_i^2\lambda_j+\sum\limits_{i<j}\lambda_i\lambda_j^2)\\
&+6\sum\limits_{i<j<k}\lambda_i\lambda_j\lambda_k\\
=&2\sum\limits_{i<j<k}(\lambda_k(\lambda_k-\lambda_i)(\lambda_k-\lambda_j)+\lambda_j(\lambda_j-\lambda_i)(\lambda_j-\lambda_k)\\
&\ \ \ +\lambda_i(\lambda_i-\lambda_k)(\lambda_i-\lambda_j))
\end{align*}
Note that $\lambda_j\leq \lambda_k$ for $j\leq k$. We have
\begin{align*}
\frac{n-2}{2}J
&=2\sum\limits_{i<j<k}(\lambda_k(\lambda_k-\lambda_i)(\lambda_k-\lambda_j)-\lambda_j(\lambda_j-\lambda_i)(\lambda_k-\lambda_j)\\
&\ \ \ +\lambda_i(\lambda_k-\lambda_i)(\lambda_j-\lambda_i))\\
&\geq 2\sum\limits_{i<j<k}(\lambda_k(\lambda_k-\lambda_i)(\lambda_k-\lambda_j)-\lambda_k(\lambda_j-\lambda_i)(\lambda_k-\lambda_j)\\
&\ \ \ +\lambda_i(\lambda_k-\lambda_i)(\lambda_j-\lambda_i))\\
&\geq  2\sum\limits_{i<j<k}(\lambda_k(\lambda_k-\lambda_j)^2
+\lambda_i(\lambda_j-\lambda_i)^2).
\end{align*}
Since $\lambda_i\geq \epsilon R$ for any $i$ and
$(\lambda_k-\lambda_j)^2
 +(\lambda_j-\lambda_i)^2\geq \frac{1}{3}((\lambda_k-\lambda_j)^2+(\lambda_k-\lambda_i)^2
 +(\lambda_j-\lambda_i)^2)$, we get
\begin{align*}
\frac{n-2}{2}J &\geq \frac{2}{3}\epsilon R \sum\limits_{i<j<k}
((\lambda_k-\lambda_j)^2+(\lambda_k-\lambda_i)^2
 +(\lambda_j-\lambda_i)^2)\\
&= \frac{2}{3}\epsilon (n-2)nR
\sum\limits_{i<j}\frac{(\lambda_i-\lambda_j)^2}{n}\\
&= \frac{2}{3}\epsilon (n-2)nR (|Rc|^2-\frac{1}{n}R^2).
\end{align*}
Hence Lemma \ref{pinching_ineq} holds immediately.
\end{proof}

Finally, we get the following improved pinching estimate.
\begin{thm}\label{pinching_estimate}
If $(M^n,g(0))$, $n\geq 3$, is a n-dimensional complete locally
conformally flat Riemannian manifold and bounded Ricci curvature
satisfying $Rc\geq \epsilon Rg>0$, then the following inequlity
holds under the Yamabe flow (\ref{yamabe})
\begin{eqnarray*}
f(t)\leq (\frac{1}{3t})^{\frac{n\epsilon}{3}},
\end{eqnarray*}
where $f$ is defined in Lemma \ref{pinching_equ} and
$\delta=\frac{n\epsilon}{3}$.
\end{thm}
\begin{proof} The assertion is trivial if $(M^n,g(t))$ is Einstein at some time.
Now by Lemma \ref{pinching_equ} and \ref{pinching_ineq}, we have
\begin{eqnarray*}
\partial_t f(t)\leq (n-1)\Delta f+\frac{2(1-\frac{n\epsilon}{3})(n-1)}{R}<\nabla
f,\nabla R>-n\epsilon Rf.
\end{eqnarray*}
Since $Rc>0$, clearly $f\leq R^{\delta}$. So we get
\begin{eqnarray*}
\partial_t f(t)\leq (n-1)\Delta f+\frac{2(1-\delta)(n-1)}{R}<\nabla
f,\nabla R>-n\epsilon f^{1+\frac{1}{\delta}}.
\end{eqnarray*}
Hence Theorem \ref{pinching_estimate} follows from maximum principle
immediately.
\end{proof}

\section{proof of Theorem \ref{main}}

Before presenting the proofs Theorem \ref{main}, we give some
remarks.

First, as we mentioned before, it is hard to control the
injectivity radius uniformly in the sequence of dilations of the
maximal solution to the Yamabe flow on noncompact manifolds
without extra conditions. Note that the curvature bound is
satisfied for the sequence of dilations in all the singularity
models. Hence, by Theorem \ref{Yamabe_precompactness},
 we have $(M_i, g_i(t), x_i)$ with positive Ricci curvature subconverges to metric space $(X,d (t), x)$ in
the sense of $C^{\infty}$-local submersions for the sequence of
dilations in all the singularity models. We have a neighborhood
$V\subset \mathbb{R}^n$ of $o$ with metric $\overline{g_V}(t)$ is
the solution to the Yamabe flow, and $(V,\overline{g_V}(t))$ modulo
an isometric pseudogroup action $\Gamma_V$ is isometric to a
neighborhood $V'$ of $x$ in the limit metric space $(X,d (t))$.
Moreover, there are maps $\varphi_i:(V,o)\to(M_i,x_i)$ such that
 $\overline{g_V}$ is the $C^{\infty}$ limit of $\varphi^{*}_ig_i$.

 Note that one difficulty in applying the methods in the proof of Theorem
 \ref{main} is that we may not apply the weak maximum principle
 directly on $V$. Fortunately, we can apply the weak maximum
 principle on $(M,g_i(t))$ and get the Harnack inequality $Z(g_i)\geq
 0$ and then $Z(\varphi^{*}_ig_i)\geq 0$. Since $\overline{g_V}$ is the $C^{\infty}$ limit of
 $\varphi^{*}_ig_i$, we still have Harnack inequality holds on
 $(V,\overline{g_V}(t))$. Similarly, Theorem \ref{pinching_estimate}
 also holds on  $(V,\overline{g_V}(t))$.

Second, we need to establish the strong maximum principle for the
Harnack quantity $Z$ in $V$. Suppose that $Z$ is positive for all
$Y\in T_{x_0}V$ at $t=t_0$, for any given point $y\in V$. Let
$\Omega\subset V$ be a connected open set such that $\bar{\Omega}$
is a compact manifold with smooth boundary and $\Omega$ contains
both $x_0$ and $y$. We can find a nonnegative function $f$ on $V$
with support on $\Omega$ so that $f(x_0)>0$ and $Z\geq
\frac{f}{t^2_0}$ for all $Y\in T_x V$ for all $x\in\Omega$ at
$t_0$. Let $f$ evolves as
\begin{equation*}
\left\{
\begin{array}{ll}
         \partial_t f=(n-1)\Delta f \quad &\text{in}\ \Omega\times [t_0,T], \\
          f(x,t)=0 &\text{on}\ \partial\Omega\times [t_0,T].
\end{array}
\right.
\end{equation*}
By the scalar strong maximum principle, we conclude that $f>0$ on
$\Omega\times (t_0,T]$. Since $(\partial_t-(n-1)\Delta)Z\geq
-\frac{2}{t}Z$(see \cite{chow}, (3.14)), we get
$(\partial_t-(n-1)\Delta)(Z-\frac{f}{t^2})\geq
-\frac{2}{t}(Z-\frac{f}{t^2})$ on $\Omega\times [t_0,T]$.
Moreover, since $Z\geq 0$, we have $Z\geq \frac{f}{t^2}$ on
$\Omega\times\{t_0\}\cup\partial \Omega\times [t_0,T]$. By the
weak maximum principle, we have $Z\geq \frac{f}{t^2}$. So $Z$ is
positive for all $Y\in T_x V$ for $x\in \Omega$ for any $t>t_0$.

Third, the arguments below show the relation between the Riemannian
neighborhood above the different points in limit space $X$, i.e. we
show that they always have the subset locally isometric to each
other if the intersection of their projection is not empty. By
Theorem \ref{Yamabe_precompactness}, we know that for all $y\in X$
there exist $y_i  \in M_i$ such that $(M^n_i,g_i(t),y_i )$ locally
converges to $(X,d(t),y)$ in the sense of $C^{\infty}$-local
submersions at $y$. Again, we emphasize that the above conclusion
holds because of Proposition 12 in \cite{DG}. Then we have a
neighborhood $U\subset \mathbb{R}^n$ of $o$ with metric
$\overline{g}(t)$ being the solution to the Yamabe flow, and
$(U,\overline{g}(t))$ modulo an isometric pseudogroup action
$\Gamma_U$ is isometric to a neighborhood $U'$ of $y$ in the limit
metric space $(X,d (t))$. Now we assume $W'=U'\cap V' \neq
\emptyset$ and define $\pi_{U}:U\to U'$, $\pi_{V}:V\to V'$. By the
definition of the Gromov-Hausdorff distance, there is a
Gromov-Hausdorff approximation map $\psi_i:X\to M_i$ such that
$\psi_i(W')$ converges to $W'$. Clearly
$\overline{g_U}|_{\pi_{U}^{-1}(W')}$ and
$\overline{g_V}|_{\pi_{V}^{-1}(W')}$ are the $C^{\infty}$ limits of
$(\varphi_U)^*_ig_i|_{\psi_i(W')}$ and
$(\varphi_V)^*_ig_i|_{\psi_i(W')}$. Hence, $\pi_{U}^{-1}(W')$ is
locally isometric to $\pi_{V}^{-1}(W')$.

With the preparations above we now give the proof of Theorem
\ref{main}.

 \textbf{Proof of Theorem \ref{main}.}
Since the sectional curvature is bounded at the initial time
$t=0$, the Yamabe flow has a solution on the complete non-compact
manifold $M^n$ in some time interval $[0,T)$.

If the singularity is of Type I, Type IIa, Type IIb. Just as
\cite{RH2}, we can take a sequence $(x_i,t_i)$ and define the
pointed rescaled solutions $(M^n,g_i(t),x_i)$,
 $t\in (\alpha_i,0]$ by letting $g_i(t) =Q_ig(t_i+Q_i^{-1}t)$,
 where $Q_i=R(x_i,t_i)$ and $\alpha_i=-t_i Q_i$, such that
  $$
 R_{g_i}(x, t)\leq C,
 $$
 for all $x\in M^n$, $t\in (\alpha_i,0]$,
 $$
R_{g_i}(x_i,0)=1,
 $$
 and
$$
t_i Q_i\to\infty.
$$
 Since the Weyl tensor of $M^n$ is vanishing and $Rc>0$, then we get
  $$
 \sup\limits_{M\times (-t_i Q_i,0]} |Rm|_{g_i}(x, t)\leq C.
 $$
  By theorem \ref{Yamabe_precompactness},
 we have $(M_i, g_i(t), x_i)$ subconverges to metric space $(X,d (t) , x)$ in
the sense of $C^{\infty}$-local submersions. Hence we have a
neighborhood $V\subset \mathbb{R}^n$ of $o$ with metric
$\overline{g_V}(t)$ is the ancient solution to the Yamabe flow, and
$(V,\overline{g_V}(t))$ modulo an isometric pseudogroup action
$\Gamma_V$ is isometric to a neighborhood $V'$ of $x$ in the limit
metric space $(X,d (t))$. Moreover, there are maps
$\varphi_i:(V,o)\to(M_i,x_i)$ such that
 $\overline{g_V}$ is the $C^{\infty}$ limit of $\varphi^{*}_ig_i$. Hence
 $|R_{\overline{g_V}}(o,0)|=1$. Applying Theorem \ref{pinching_estimate} on time interval $[-\alpha,0]$,
 we get
\begin{equation}\label{ancient_est}
(|Rc|^2-\frac{1}{n}R^2)(\overline{g_V}(0))\leq
\frac{R^{2-\frac{n\epsilon}{3}}(\overline{g_V}(0))}{(3\alpha)^{\frac{n\epsilon}{3}}}.
\end{equation}
Since $\overline{g_V}(t)$ is an ancient solution, letting
$\alpha\to\infty$, we get
$(|Rc|^2-\frac{1}{n}R^2)(\overline{g_V}(0))\equiv 0$. Then this
implies $Rc_{\overline{g_V}(0)}\equiv c_1>0$ in $V$. Since the Weyl
tensor of $\overline{g_V}$ is vanishing, we conclude that
$sec_{\overline{g_V}(0)}\equiv c_2>0$ in $V$.

 By Theorem \ref{Yamabe_precompactness}, we know that for all $y\in
X$ there exists $y_i  \in M^n_i$ for each $M^n_i$ such that
$(M^n_i,g_i(t),y_i )$ locally converges to $(X,d(t),y)$ in the sense
of $C^{\infty}$-local submersions at $y$. Then we have a
neighborhood $U\subset \mathbb{R}^n$ of $o$ with the metric
$\overline{g_U}(t)$ being the ancient solution to the Yamabe flow,
and $(U,\overline{g_U}(t))$ modulo an isometric pseudogroup action
$\Gamma_U$ is isometric to a neighborhood $U'$ of $y$ in the limit
metric space $(X,d (t))$. Now we assume $W'=U'\cap V' \neq
\emptyset$ and define $\pi_{U}:U\to U'$, $\pi_{V}:V\to V'$. As we
noticed before, $\pi_{U}^{-1}(W')$ is locally isometric to
$\pi_{V}^{-1}(W')$. Then
$sec_{\overline{g}_{\pi_{U}^{-1}(W')}(0)}=sec_{\overline{g}_{\pi_{V}^{-1}(W')}(0)}\equiv
c_2>0$. Now repeat the same arguments before, we can conclude
$sec_{\overline{g}_{U}(0)}=sec_{\overline{g}_{V}(0)}\equiv c_2>0$.

Hence clearly we have for all the point in $X$ there exists a
neighborhood isometric to a Riemannian neighborhood, which has
constant curvature, modula a pesudogroup action. Furthermore, the
curvature Riemannian neighborhood in different point has the same
value $c_2>0$. Then $X$ is an Alexandrov space with curvature$\geq
c_2>0$ by the Corollary in \cite{BGP92} (see $\S 4.6$, in Page 16).
Then $X$ must be compact by Theorem \ref{Alexandrov}, which is a
contradiction.

If the singularity is of Type III, i.e.
$\sup\limits_{M\times[0,\infty)}tR<\infty$. Set
$A=\limsup\limits_{t\to\infty}tM(t)$, where
$M(t)=\sup\limits_{M}R(x,t)$.  By B.Chow's Harnack inequality
(\ref{Benchow's harnack}) and taking $X=0$, we get $\frac{\partial
}{\partial t}(tR)\geq 0$. Hence we have $A>0$. So we can take a
sequence $(x_i,t_i)$ such that $t_i\to\infty$ and
 $A_i\doteq t_iR(x_i,t_i)\to A$.
 Define the pointed rescaled solutions $(M^n,g_i(t),x_i)$,
 $t\in (-t_i R_i,\infty)$, by $g_i(t) =Q_ig(t_i+Q_i^{-1}t)$, where
 $Q_i=R(x_i,t_i)$. For any $\epsilon>0$ we can find a time $\tau<\infty$
such that for $t\geq \tau$ and any $x\in M^n$
$$
tR(x,t)\leq A+\epsilon.
$$
 Then we have
  $$
 R_{g_i}(x, t)\leq \frac{A+\epsilon}{A_i+t},
 $$
for all $x\in M^n$, $t\in [-\frac{A_i(t_i-\tau)}{t_i},\infty)$ and
 $$
R_{g_i}(x_i,0)=1.
 $$

Set $\phi_i=exp_{x_i,g_i(0)}$ and $B(o_i,\frac{\pi}{2})\subset
T_{x_i}M$ equipped with metric $\widetilde{g}_i(t)\triangleq
\phi_i^*g_i(t)$. By Lemma \ref{precompactness}, we get
$(B(o_i,\frac{\pi}{2}),\widetilde{g}_i(t),o_i)$ subconverges to a
Yamabe flow $(B(o,\frac{\pi}{2}),\widetilde{g}(t),o)$ in
$C^{\infty}$ sense. Hence
$$ R_{\widetilde{g}(t)}(x,t)\leq \frac{A}{A+t}$$
for all $x\in B(o,\frac{\pi}{2})$, $t\in (-A,\infty)$ and
$$R_{\widetilde{g}(t)}(o,0)=1.$$

Then by Theorem \ref{expanding_soliton}, we conclude that
$(B(o,\frac{\pi}{2}),\widetilde{g}(t))$ is an expanding soliton of
Yamabe flow, i.e. there is smooth vector field satisfying
$\nabla_k X^i-(R+\frac{1}{t})g_{ik}=0$. Moreover, $X$ is the
unique solution of the equation
\begin{equation}\label{ddddd}
\nabla_i R+\frac{1}{n-1}R_{ij}X^j= 0.
\end{equation}

We use the arguments due to A.Chau and L.F.Tam \cite{CT}(see
Theorem 2.1) to show the injectivity radius of $x_i$ have the
uniformly lower bound with respect to $g_i(0)$. By our assumptions on the positivity of Ricci
curvature, we may then let $W(i) \in TM$ be the unique solutions to
(\ref{ddddd}) on $(M^n, g_i(t))$ for any $i$. Set
$V(i)=\phi_i^*W(i)$. Then $V(i)$ converges to $X$ in the
$C^{\infty}$ sense.

In some coordinates $x^{\alpha}$ of $B(o,\frac{\pi}{2})$, the
integral curves of $-X(\cdot,0)$ (i.e. the vector field $X$ at time
$t=0$) are given by the following
$$
x'_{\alpha} = -\lambda_{\alpha}x_{\alpha} + F_{\alpha}(x)
$$
where $\lambda_{\alpha}\geq c
> 0$ are the positive eigenvalues of $(R+\frac{1}{t})g(\cdot,0)$, $|F(x)| =
O(|x|^2)$ and $|dF(x)| = O(|x|)$. For any $\epsilon > 0$, there
exists sufficient larger $i$ and $0<r_1<\frac{\pi}{2}$ such that the
integral curves of $-V(i)(\cdot,0)$ is given by
$$
 x'_{\alpha} = -\lambda_{\alpha}x_{\alpha} + G^i_{\alpha}(x),
 $$
with $|G^i-F|+ |dG^i -dF|\leq \epsilon $ in $B_{g_i}(o,r_1)$.

Let $x(\tau)$ be an integral curve of $-V(i)(\cdot,0)$ in
$B_{g_i}(o,r_1)$. Set $|x|\leq r_2<r_1$, where $r_2$ is to be
determined later. We calculate
\begin{align*}
\frac{d}{d\tau}|x|^2
&\leq -2c|x|^2+|G^i||x|\\
&\leq -2c|x|^2+\epsilon |x|+C_1|x|^2\\
&\leq -\frac{3c}{2}|x|^2+\epsilon |x|,
\end{align*}
where $C_1>0$ is a constant only depending on $F$ and
$r_2<\frac{c}{2C_1}$. Then if $\frac{r_2}{2}\leq |x| \leq r_2$, we
have $\frac{d}{d\tau}|x|^2<0$ if $\epsilon$ is sufficient small.
Hence for $i$ large enough any integral curve of $-V(i)(\cdot,0)$
starting in $B_{g_i}(o,r_2)$ will stay inside $B_{g_i}(o,r_2)$.

Now Let $x(\tau)$ and $y(\tau)$ be two integral curves of
$-V(i)(\cdot,0)$ inside $B_{g_i}(o,r_2)$. Then we calculate
\begin{align*}
\frac{d}{d\tau}|x-y|^2
&\leq -2c|x-y|^2+||dG^i|||x-y|^2\\
&\leq -c|x-y|^2.
\end{align*}
Hence
\begin{equation}\label{eeeeee}
|x-y|(\tau)\leq \exp(-c\tau)|x-y|^2(0)\leq
4r^2_2\exp(-c\tau).
\end{equation}
 Set $y(\tau)=x(\tau_2-\tau_1+\tau)$. Then we have
$|x(\tau_1)-y(\tau_1)|=|x(\tau_1)-x(\tau_2)|\leq 4r^2_2\exp(-c\tau_1)$.
 Hence $x(\tau)$ converges to a point $x_0\in B_{g_i}(o,r_2)$. By (\ref{eeeeee}), we conclude that
 $y(\tau)$ also converges to $x_0$.

 Next we show $\phi_i$ is injective on $B_{g_i}(o,r_2)$ for $i$ sufficient large, which imply that $inj(g_i(0),x_i)\geq r_2$.
 Otherwise, there exist two points
 $p_1 \neq p_2 \in B_{g_i}(o,r_2)$ such that $\phi_i(p_1)=\phi_i(p_2)=q\in M^n$. Let $\gamma_1$ and $\gamma_2$
 be two integral curves for $-V(i)(\cdot,0)$ starting
at $p_1$ and $p_2$ respectively. Hence $\phi_i(\gamma_1)$ and $\phi_i(\gamma_2)$
 be two integral curves for $-W(i)(\cdot,0)$ starting
at $q$. By uniqueness of the integral curves,
we have $\phi_i(\gamma_1(\tau))=\phi_i(\gamma_2(\tau))$ for all $\tau$.
On the other hand, for all $\tau$, we also have $\gamma_1(\tau ) \neq \gamma_2(\tau)$ by
uniqueness of integral curves.
But $\gamma_1$ and $\gamma_2$ both converge to the point
$x_0\in B_{g_i}(o,r_2)$. It contradicts the fact that
$\phi$ is the diffeomorphism in some neighborhood of $x_0$.

Then we have $(M^n,g_i(t),x_i)$ subconverges to a noncompact
 Type III limit solution to Yamabe flow (an expanding soliton) with $Rc\geq \epsilon Rg>0$ which contradicts to Theorem \ref{pinching_expanding}.

 This completes the proof of Theorem \ref{main}.
 $\Box$

\thanks{\textbf{Acknowledgement}: The second author would like to
thank Dr. Qingsong Cai for helpful discussion.}

\end{document}